\documentclass{article}

\usepackage{graphics}
\usepackage{graphicx}
 \usepackage{epsfig}
\usepackage{amssymb}
 \usepackage{amsthm}
\usepackage[T1]{fontenc}
\usepackage[latin1]{inputenc}
\usepackage{lmodern}
\usepackage{amsfonts}
\usepackage{amssymb}
\usepackage{amsmath}
\usepackage{enumerate}
\usepackage[english]{babel}
\usepackage{wasysym}
\usepackage{xspace}
\usepackage[figuresright]{rotating}
\usepackage{lineno}

\newtheorem{theorem}{Theorem}

\newtheorem{lemma}[theorem]{Lemma}

\title{A sufficient condition for the existence of plane spanning trees on geometric graphs\thanks{Partially supported by Conacyt, M{\'e}xico.} }

\author{Eduardo Rivera-Campo\thanks {Corresponding author.} 
and Virginia Urrutia-Galicia\thanks {Departamento de Matem\'aticas, Universidad Aut\'onoma Metropolitana-Iztapalapa, Av. San Rafael Atlixco 186, M{\'e}xico D.F., C.P. 09340, \{erc, vug\}@xanum.uam.mx.}}

\date{}

\begin{document}
\maketitle

\begin{abstract}
\noindent Let $P$ be a set of $n \geq 3$ points in general position in the plane and let $G$ be a geometric graph with vertex set $P$.
If the number of empty triangles $\bigtriangleup u v w$ in $P$ for which the subgraph of $G$ induced by $\{u,v,w\}$
is not connected is at most $n-3$, then $G$  contains a non-self  intersecting spanning tree.

\end{abstract}

\begin{quotation}
\end{quotation}

\section{Introduction}

Throughout  this article $P$ denotes a set of $n\geq 3$ points in general position in the Euclidean
plane. A geometric graph with vertex set $P$ is a graph $G$ drawn in such a way that each edge is a straight line
segment with both ends in $P$. A \emph{plane spanning tree} of $G$
is a non-self  intersecting subtree of $G$ that contains every
vertex of $G$. Plane spanning trees with or without specific conditions have been studied by various authors.

A well known result of K{\'a}rolyi  \emph{et al}  \cite{KaPaTo}
asserts that if the edges of a finite complete geometric graph $GK_n$  are coloured by two colours, then there exists a plane spanning tree of $GK_n$ all of whose edges are of the same colour. Keller  \emph{et al}  \cite{KePeRiUr}  characterized those plane spanning trees $T$ of $GK_n$ such that the complement graph $T^c$ contains no plane spanning trees.

A plane spanning tree $T$ is a  \emph{geometric independency  tree} if for each pair $\{ u,v \} $ of leaves of $T$, there is an edge  $xy$ of  $T$ such that the segments $uv$ and $xy$ cross each other. Kaneko  \emph{et al}  \cite{KaOdYo}  proved  that every complete geometric graph with $n \geq 5$ vertices contains a geometric independency tree with at least  $\frac{n}{6}$  leaves.

Let $k$ be an integer with $2 \leq k \leq 5$ and $G$ be a geometric graph with $n \geq k$ vertices such that  all geometric subgraphs of $G$ induced by $k$ vertices have a plane spanning tree. Rivera-Campo  \cite{Ri} proved that $G$ has a plane spanning tree.

Three points $u, v$ and $w$ in $P$ form an \emph{empty triangle}  if no point of $P$ lies in the interior of the triangle
$\bigtriangleup u v w$.  For any geometric graph $G$ with vertex set $P$ we say that an empty triangle $\bigtriangleup u v w$ of $P$  is \emph{disconnected} in $G$ if the subgraph of $G$ induced by $\{u, v, w\}$ is not connected.

Let  $s(G)$ denote the number of disconnected empty triangles of $G$. Our result is the following:

\begin{theorem} \label{mainresult}

If $G$ is a geometric graph with $n \geq 3$ vertices such that $s(G) \leq
n-3$, then $G$ has a plane spanning tree.

\end{theorem}

For each $n \geq 3$, let $u_1, u_2, \dots, u_n$ be the vertices of
a regular $n$-gon and denote by $T_n$ and $T_n^c$ the plane path $ u_1, u_2,
\dots,  u_n$ and its complement, respectively. The geometric graph  $T_n^c$ contains no plane spanning tree and is such that $s(T_n^c) =n-2$. This shows that the condition in Theorem~\ref{mainresult} is tight.

\section{Proof of Theorem~\ref{mainresult}}

For every oriented straight line $L$ we denote by $L^-$ the set of
points in $P$ which are on or to the left of $L$ and by $L^+$ the
points which are on or to the right of $L$.

A $k$-set of $P$ is a subset $X$ of $P$ with $k$ elements that can be obtained by intersecting $P$ with an open half plane. The main tool in the proof of Theorem~\ref{mainresult} is the
following procedure of Erd\H{o}s \emph{et al}  \cite{Lo},  \cite{ErLoSiSt}, 
used to generate all $k$-sets of $P$: Let  $L=L_1$  be an oriented line passing through precisely one point $v_1$ of $P$ with $|L^-_1|=k+1$ .  Rotate $L$ clockwise around the
axis $v_1$ by an angle $\theta$  until  a point $v_2$ in $P$ is
reached. Now rotate $L$ in the same direction but around
$v_2$ until a point $v_3$ in $P$ is reached, and continue rotating
$L$ in a similar fashion obtaining a set of oriented lines $C(L)$
and a sequence of points $v_1, v_2, \ldots, v_s$, not necessarily
distinct, where  $v_s= v_1$ when the angle of rotation $\theta$
reaches $2\pi$.

For $i=1,2, \ldots, s-1$, let $L(v_i, v_{i+1})$ be the line in
$C(L)$ that passes through points $v_i$ and $v_{i+1}$ and for
$i=2,3, \ldots, s-1$, let $L_i$ be any line in $C(L)$ between
$L(v_{i-1}, v_i)$ and $L(v_i, v_{i+1})$.

It is well known that for each line $L_j$ either $L^+_{j+1}= L^+_j$
and $L^-_{j+1}= (L^-_j \setminus \{v_j\}) \cup \{v_{j+1}\}$, or
$L^+_{j+1}= (L^+_j \setminus \{v_j\}) \cup \{v_{j+1}\}$ and
$L^-_{j+1}= L^-_j$. In both cases $|L^-_{j+1}|= |L^-_j|=k+1$ and  $|L^+_{j+1}|= |L^+_j|=n-k$. 
It is also easy to see that if $v_{j+1} \in
L^+_j$, then $L^-(v_j , v_{j+1})= L^-_j  \cup \{v_{j+1}\}$ and
$L^+(v_j , v_{j+1})= L^+_j$, and if  $v_{j+1} \in L^-_j$, then
$L^-(v_j , v_{j+1})= L^-_j$ and $L^+(v_j , v_{j+1})= (L^+_j
\setminus \{v_j\}) \cup \{v_{j+1}\}$.

The following lemma will used in the proof of Theorem~\ref{mainresult}.

\begin{lemma} \label{crosses}

Let $L_i, L_j \in C(L)$ with $i < j $. If $x$,  $y$ and $z$ are
points of $P$ lying in $L^+_i \cap L^-_j$, then there are integers
$k$ and $l$ with $i \leq k < l < j$ such that $v_k \in \{x,y,z\}$,
$x,y,z \in L^+_k \cap L^-_j$ and such that $L_l$ crosses the
triangle $\bigtriangleup x y z$.

\end{lemma}

\begin{proof}

Consider the lines $L_i, L_{i+1}, \ldots, L_j$. The result follows
from the fact that at each step $t$, at most one of the points
$x,y,z$ switches from $L^+_{t}$ to $L^-_{t+1}$.  See Fig. \ref{lema}

\end{proof}

\begin{figure}[h!!]
\begin{center}
\includegraphics{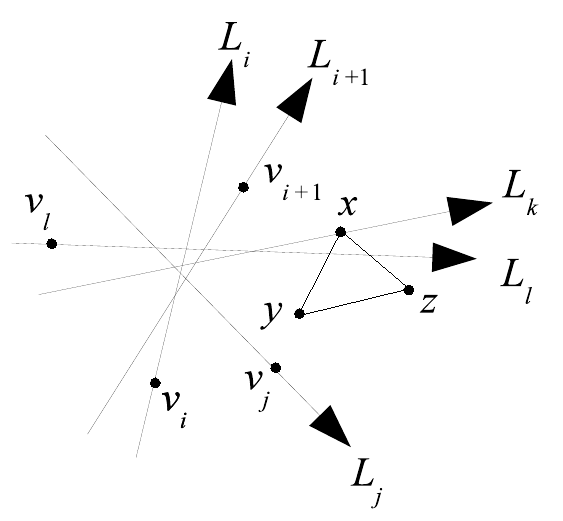}
\caption{$x, y, z \in  L^+_i  \cap L^+_{i+1} \cap \ldots \cap  L^+_k $, $x=v_k$ and $L_l$ crosses  $\{x,y,z\}$. }
\label{lema}
\end{center}
\end{figure}


Let $G$ be a geometric graph with $n\geq3$ vertices such that $s(G)
\leq n-3$ and let $P$ denote the vertex set of $G$. If $n=3$ or $n=4$, it is not
difficult to verify by inspection that $G$ has a plane spanning tree. Let us
proceed with the proof of Theorem~\ref{mainresult} by induction and
assume $n \geq 5$ and that the result is valid for each geometric
subgraph of $G$ with $k$ vertices, where $3 \leq k \leq n-1 $.

Let $v_1$ be a point in $P$ and $L_1$ be an oriented line through
$v_1$ such that $|L_1^-| =\lceil \frac{n+1}{2} \rceil$ and $|L_1^+|
= \lfloor \frac{n+1}{2} \rfloor$. Let $C(L)$ be the set of oriented
lines obtained from $L=L_1$ as above.

For every $i \ge 1$, define $G^-_i$ and $G^+_i$ as
the geometric subgraphs of $G$ induced by $L^-_i$ and $L^+_i$ respectively, and
$G^-(v_{i}, v_{i+1})$ and $G^+(v_{i}, v_{i+1})$ as
the geometric subgraphs of $G$ induced by $L^-(v_{i}, v_{i+1})$ and $L^+(v_{i}, v_{i+1}))$,
respectively.

We show there is a line in $C(L)$ for which induction applies to the corresponding graphs $G^-$ and $G^+$,
giving plane spanning trees $T^-$ of $G^-$ and $T^+$ of $G^+$. As $T^-$ and $T^+$ lie in opposite sides of $L$,
their union contains a plane spanning tree of $G$. We analyse several cases.

\medskip

\noindent \textbf{Case 1.} $s(G^{-}_1) \leq |L^{-}_1| - 3$ and
$s(G^{+}_1) \leq |L^{+}_1| - 3$.

By induction there exist plane spanning trees $T^-_1$ of $G_1^-$ and $T^+_1$ of
$G_1^+$.  Since $T^-_1$ and $T^+_1$ lie in opposite sides of $L_1$ and contain exactly one point in common,
the graph $T^-_1 \cup T^+_1$ is a plane spanning tree of $G$.

\medskip
\noindent \textbf{Case 2.} $s(G^{-}_1) \geq |L^{-}_1| - 2$ and
$s(G^{+}_1) \geq |L^{+}_1| - 2$.

Clearly  $s(G^{-}_1) + s(G^{+}_1) \geq  (|L^{-}_1|-2) +
(|L^{+}_1|-2) = n-3 \geq s(G) \geq s(G^{-}_1) + s(G^{+}_1).$ This implies $s(G^{-}_1) = |L^{-}_1| - 2$, $s(G^{+}_1) = |L^{+}_1| - 2$
and that $L_1$ does not cross any disconnected empty triangle of $G$.

Consider the line $L_m$ in $C(L)$ parallel to $L_1$ with opposite
orientation. As $L_1^+ \subset L_m^-$, any disconnected empty
triangle of $G^+_1$ is also a disconnected empty triangle of
$G^-_m$. By Lemma \ref{crosses}, there exists a line in $C(L)$ that
crosses a disconnected empty triangle of $G$. Let $j$ be the
smallest integer such that $L_{j+1}$ crosses a disconnect empty
triangle $\bigtriangleup x y z$ of $G$.

Since  $L_1, L_2,\ldots, L_j$ do not cross any disconnected empty
triangle of $G$, it follows that $s(G^{-}_j) = s(G^{-}_1)= |L^{-}_1|
- 2= |L^{-}_j| - 2$ and that $s(G^{+}_j) = s(G^{+}_1)= |L^{+}_1| -
2= |L^{+}_j| - 2$. Moreover, also by  Lemma \ref{crosses}, the axis vertex  $v_j$ of $L_j$ must be
one of the vertices $x$, $y$, or $z$, since $L_{j+1}$ crosses
$\bigtriangleup x y z$ while $L_1, L_2,\ldots, L_j$ do not. Without
loss of generality we assume $z = v_j$. See Fig.~\ref{crosses-mod}.

\begin{figure}[h!!]
\begin{center}
\includegraphics{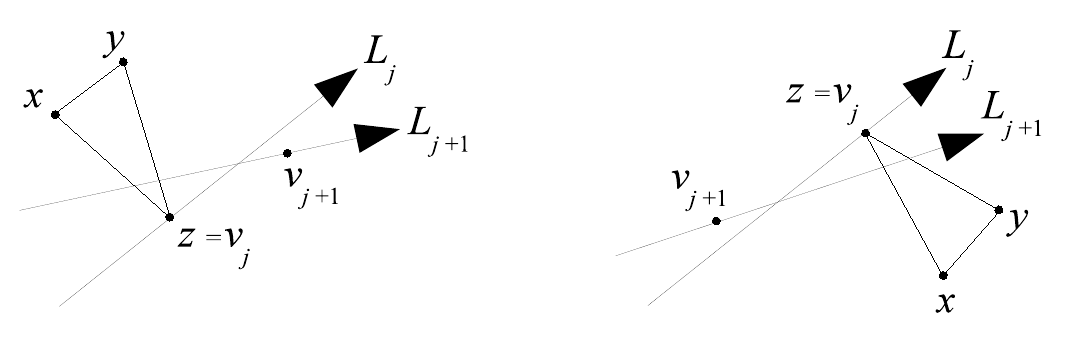}
\caption{$L_1,L_2, \ldots , L_j$ do not cross any disconnected empty
triangle of $G$ and $L_{j+1}$ crosses a disconnected empty triangle
$\bigtriangleup x y v_j$ of $G$. } \label{crosses-mod}
\end{center}
\end{figure}

\medskip

\noindent \textbf{Case 2.1.} $v_{j+1} \in L_j^+$.

In this case $\bigtriangleup x y z$ is a disconnected empty triangle
of $G^-_j$, see Fig.~\ref{crosses-mod} (left) and Fig.~\ref{derecha}. Let $i \geq j+1$ be the smallest
integer such that the axis vertex $v_{i+1}$ of  $L_{i+1}$ lies in
$L_{j}^-$. By the choice of $i$, all points $v_{j+1}, v_{j+2},
\ldots, v_{i}$ lie in $L_j^+$ and therefore $L_i^+ =L_{i-1}^+ =
\dots = L_{j}^+$. It follows that $G^+_i=G^+_{i-1}= \dots =G^+_j$
and that $s(G^+_i)=s(G^+_{i-1})= \dots = s(G^+_{j})$.

\begin{figure}[h!!]
\begin{center}
\includegraphics[height=7cm, width=11cm]{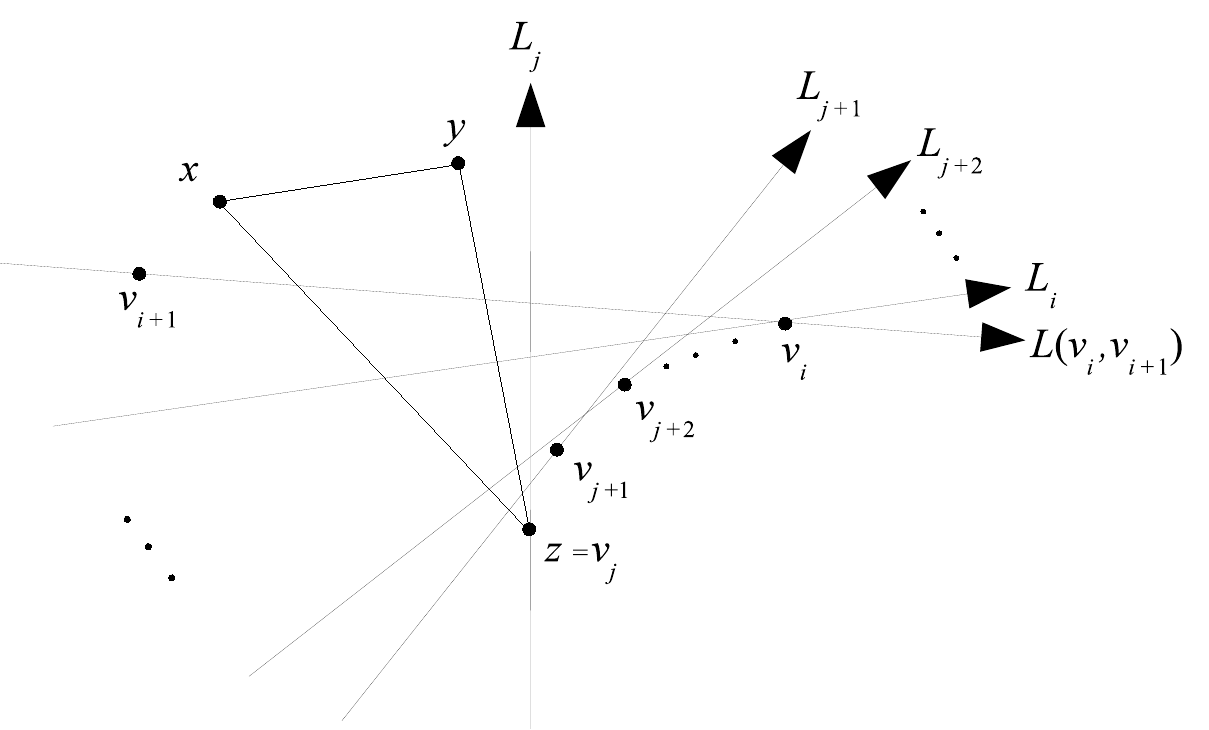}
\caption{$v_{j+1}, v_{j+2}, \ldots, v_{i}  \in L_1^+$, $L_{i+1} \in L_{1}^-$.}
\label{derecha}
\end{center}
\end{figure}

Again by the choice of $i$, $L_{k+1}^- =(L_k^- \setminus \{v_k\})
\cup  \{v_{k+1}\}$ for $k=j,j+1,\ldots,i-1$ and therefore $L_i^-
=(L_j^- \setminus \{v_j\}) \cup  \{v_i\}$. Moreover, all lines
$L_{i}, L_{i-1}, \dots, L_{j+1}$ cross $\bigtriangleup x y z$. This
implies $s(G^-_i)\leq s(G^-_j)-1$  since $\bigtriangleup x y z$ is a
disconnected empty triangle of $G^-_j$.

Now consider the line $L(v_{i},v_{i+1})$  and notice that
$L^{-}(v_{i},v_{i+1}) =  L^{-}_{i} $ and $L^{+}(v_{i},v_{i+1}) =
L^{+}_{i} \cup \{v_{i+1}\}$ because $v_{i+1} \in L_{i}^-$. Therefore
 $$|L^{-}(v_{i},v_{i+1})| =  |L^{-}_{i}| = |L^{-}_{j}| \quad \mbox{and} \quad  |L^{+}(v_{i},v_{i+1})| = |L^{+}_{i}| + 1 = |L^{+}_{j}| + 1$$

Also notice that $s(G^{-}(v_{i},v_{i+1})) = s(G^{-}_{i})$ and
$s(G^{+}(v_{i},v_{i+1})) = s(G^{+}_{i})$ because no empty triangle
of $G$ contained in $L^{+}(v_{i},v_{i+1})$ has $v_{i+1}$ as one of
its vertices since $L_j$ does not cross any empty triangle of $G$.
Therefore
$$s(G^{-}(v_{i},v_{i+1})) = s(G^{-}_{i})= s(G^{-}_{j})-1 =(|L^{-}_{j}| -2)-1=|L^{-}_{j}| -3=
|L^{-}(v_{i},v_{i+1})|-3$$
$$\quad \mbox{and} \quad$$
$$s(G^{+}(v_{i},v_{i+1})) = s(G^{+}_{i}) = s(G^{+}_{j})= |L^{-}_{1}| -2 = (|L^{+}(v_{i},v_{i+1})|-1)-2= |L^{+}(v_{i},v_{i+1})|-3.$$

By  induction, there exist plane spanning trees $T^-$ of $G^-(v_{i}, v_{i+1})$
and $T^+$ of $G^+(v_{i}, v_{i+1})$. The theorem follows since
$T^-\cup T^+$ contains a plane
spanning tree of $G$.

\medskip

\noindent \textbf{Case 2.2.} $v_{j+1} \in L_j^-$.

In this case $\bigtriangleup x y z$ is a disconnected
empty triangle of $G^+_j$, see Fig.~\ref{crosses-mod} (right). The proof is analogous to that of
Case 2.1.

\medskip

\noindent \textbf{Case 3.} $s(G^{-}_1)\geq |L^{-}_1| - 2$ and
$s(G^{+}_1) \leq |L^{+}_1| - 3$.

If for every $L_j \in C(L)$, $$s(G^{-}_j) \geq |L^{-}_j| - 2 \quad
\mbox{and} \quad s(G^{+}_j) \leq |L^{+}_j| - 3,$$ then for $L_m$ in
particular, the line in $C(L)$ parallel to $L_1$ with the
opposite orientation, we have that $$ s(G^{-}_m) \geq |L^{-}_m| - 2
\quad \mbox{and} \quad s(G^{+}_m) \leq |L^{+}_m|- 3. $$

If $n$ is odd, then $L_1$ and $L_m$ are the same line but with
opposite orientations, in which case $L^{-}_m=L^{+}_1$ and
$L^{+}_m=L^{-}_1$. It follows that
$$|L^{+}_m|- 2 = |L^{-}_1|- 2 \leq s(G^{-}_1) = s(G^{+}_m) \leq |L^{+}_m|- 3,$$
which is not possible.

If $n$ is even, then $L_1$ and $L_m$ are parallel lines with
opposite orientations, with $L_m$ to the left of $L_1$ and with $|L^{+}_1|+|L^{+}_m| = n$. This
implies that there are no points between $L_1$ and $L_m$. Therefore every
empty triangle of $G^-_1$ contains points in $L^+_m$ and every empty triangle of
$G^-_m$ contains points in $L^+_1$. Thus no empty triangle of  $G^-_1$
is also an empty triangle of  $G^-_m$, see Fig~\ref{triplets}.

\begin{figure}[h!!]
\begin{center}
\includegraphics{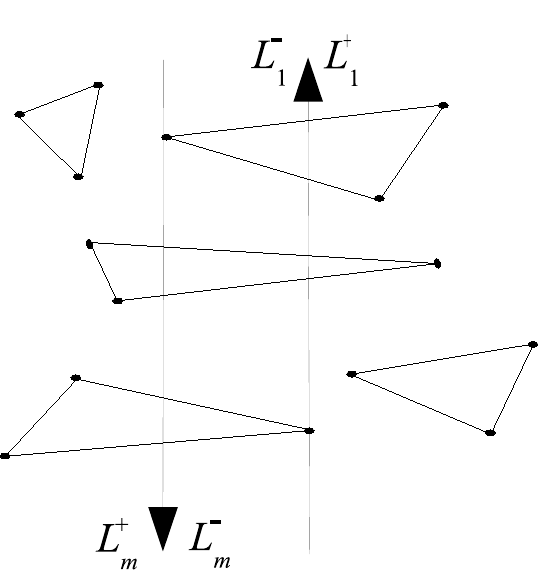}
\caption{No empty triangle of $G$ is contained in $L^-_1 \cap L^-_m$.}
\label{triplets}
\end{center}
\end{figure}

It follows that $s(G^{-}_1) + s(G^{-}_m) \leq s(G)$ which is also a
contradiction since
$$s(G) \leq n-3 < n -2 = |L^{-}_1|- 2 + |L^{-}_m|- 2 \leq s(G^{-}_1) + s(G^{-}_m).$$

Therefore, there exists $L_k \in C(L)$ such that
$$s(G^{-}_k)\geq |L^{-}_k| - 2 \quad \mbox{and} \quad s(G^{+}_k) \leq |L^{+}_k| - 3,$$
while
$$s(G^{-}_{k+1})\leq |L^{-}_{k+1}| - 3 \quad \mbox{or} \quad s(G^{+}_{k+1}) \geq |L^{+}_{k+1}| - 2.$$

Since $L^-_{k+1}=L^-_k$ or $L^+_{k+1}=L^+_k$, it must happen that  either

$$s(G^{-}_{k+1}) \leq |L^{-}_{k+1}| - 3 \quad \mbox{and} \quad s(G^{+}_{k+1}) \leq |L^{+}_{k+1}| - 3$$
$$\quad \mbox{or} \quad$$
$$s(G^{-}_{k+1}) \geq |L^{-}_{k+1}| - 2 \quad \mbox{and} \quad s(G^{+}_{k+1}) \geq  |L^{+}_{k+1}| - 2$$
which are Case 1 and Case 2, respectively.

\medskip

\noindent \textbf{Case 4.}  $s(G^{-}_1) \leq |L^{-}_1| - 3$ and
$s(G^{+}_1) \geq |L^{+}_1| - 2$

As above, let $L_m$ be the line in $C(L)$ parallel to $L_1$ with opposite orientation. If  $n$ is odd, then $L^{-}_m=L^{+}_1$ and
$L^{+}_m=L^{-}_1$. Therefore $s(G^{-}_m)\geq |L^{-}_1| - 2$ and
$s(G^{+}_1) \leq |L^{+}_m| - 3$ which is Case 3.

For $n$ even additional considerations are needed. For the sake of completeness we include the entire proof for this subcase.

If $s(G^{-}_j) \leq |L^{-}_j| - 3$ and $s(G^{+}_j) \geq
|L^{+}_j| - 2$ for every $L_j \in C(L)$, then,   $s(G^{-}_m) \leq
|L^{-}_m| - 3 = l-2$ and $s(G^{+}_m) \geq |L^{+}_m| - 2$. As $L^{+}_1 \subset L^{-}_m$ and $L^{+}_m \subset L^{-}_1$, we have,
 $$|L^{+}_1| - 2 \leq s(G^{+}_1) \leq s(G^{-}_m) \leq |L^{-}_m| - 3 = |L^{+}_m| - 2$$
 $$\quad \mbox{and} \quad$$

 $$|L^{+}_m| - 2 \leq s(G^{+}_m) \leq s(G^{-}_1)\leq |L^{-}_1| - 3 = |L^{+}_1| - 2$$
which implies  $s(G^{+}_1) = s(G^{-}_1) = s(G^{+}_m) = s(G^{-}_m)$,
since $|L^{+}_1|=|L^{+}_m|$.

It follows that no disconnected empty triangle $\bigtriangleup x y z$ of
$G^{-}_1$ has $v_1$ as one of its vertices, otherwise $L_m$ must
cross $\{x,y,z\}$ in which case $s(G^+_m)<s(G^-_1)$ because $G^+_m$
is a subgraph of $G^-_1$.

By our assumption, the same argument can be applied to every line $L_j $ in
$C(L)$ and therefore  for each graph  $G^{-}_j$,  no disconnected
empty triangle of $G^{-}_j$ has $v_j$ as one of its vertices.

To reach a contradiction consider any disconnected empty triangle
$\bigtriangleup x y z$  of $G^{+}_1$.  As $L_m$ is parallel to $L_1$ and to the
left of $L_1$, then $\bigtriangleup x y z$ is also a disconnected empty triangle
of  $G^{-}_m$ and therefore $\bigtriangleup x y z$ lies to the right of $L_1$ and to
the left of $L_m$. By Lemma \ref{crosses}, there is a line $L_t$ in
$C(L)$ with $1 < t < m$ such that $\bigtriangleup x y z$ is a disconnected
empty triangle of $G^{+}_t$ and one of its vertices is precisely
$v_t$, which is the contradiction, see Fig.~\ref{tangente}.

\begin{figure}[h!!]
\begin{center}
\includegraphics[height=7cm, width=11cm]{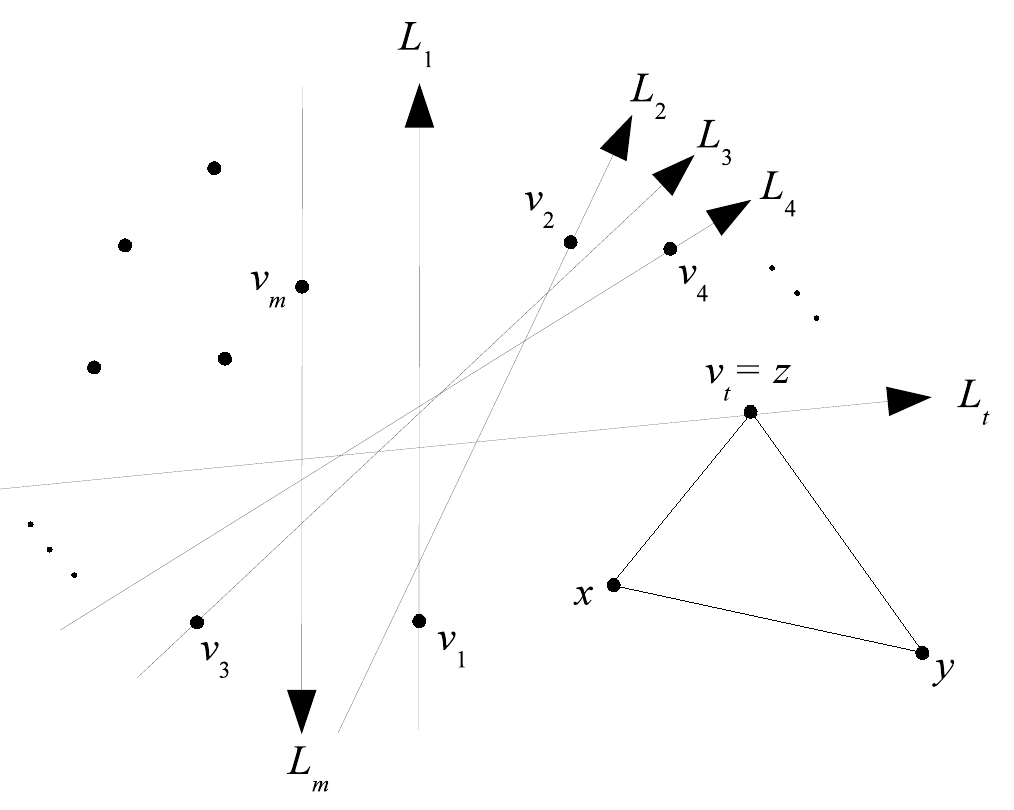}
\caption{$v_t = z$.} \label{tangente}
\end{center}
\end{figure}

As in Case 3, there is a line $L_k$ in $C(L)$ such that
$$s(G^{-}_k) \leq |L^{-}_k| - 3 \quad \mbox{and} \quad s(G^{+}_k)
\geq |L^{+}_k| - 2,$$
while
 $$s(G^{-}_{k+1}) \geq |L^{-}_{k+1}| - 2
\quad \mbox{or} \quad s(G^{+}_{k+1}) \leq |L^{+}_{k+1}| - 3 $$

Again, since $L^-_{k+1}=L^-_k$ or $L^+_{k+1}=L^+_k$, it must happen
that either
$$s(G^{-}_{k+1}) \leq |L^{-}_{k+1}| - 3 \quad
\mbox{and} \quad s(G^{+}_{k+1}) \leq |L^{+}_{k+1}| - 3$$
 $$\quad \mbox{or} \quad$$
$$s(G^{-}_{k+1}) > |L^{-}_{k+1}| - 3 \quad \mbox{and} \quad
s(G^{+}_{k+1}) >|L^{+}_{k+1}| - 3$$ which are Case 1 and Case 2,
respectively. This ends the proof of Theorem~\ref{mainresult}.

\hfill $ \qed $

\section{Final Remark}

For $n \geq 5$, let $v_1, v_2, \dots, v_{n-1}$ be the vertices of a
regular $(n-1)$-gon and let $w$ be a point closed to $v_{n-1}$ and
in the interior of the triangle $\triangle v_{n-3}v_{n-2}v_{n-1}$.
Denote by $R_n$ and $R^c_n$ the plane path $v_1, v_2, \dots, v_{n-1}, w$ and its complement, respectively. The
geometric graph $R^c_n$ is such that $s(R^c_n) = n-3$ and both
graphs $R_n$ and $R^c_n$ contain plane spanning trees. This shows
that Theorem~\ref{mainresult} is not (at least not an immediate)
consequence  of the  result by K\'{a}rolyi \emph{et al} mentioned
above.

\section{Acknowledgments}

We thank the anonymous referees for their suggestions that help us to improve the organisation and readability of the paper.

\end{document}